\documentclass{amsart}
\usepackage{graphicx}  
\usepackage{amsmath}
\usepackage{mathtools}
\usepackage{amssymb}
\usepackage{bbm}
\usepackage{hyperref}
\usepackage{mathdots}
\usepackage{dsfont}
\usepackage{tikz}
\usetikzlibrary{calc}
\usepackage{subcaption}
\usepackage{comment}
\usepackage{cleveref}
\usepackage{amsthm}

\def\det{{\mathrm{det}}}
\def\T{{\top}}

\def\type{{\rm prof}}

\def\F{{\mathbb F}}

\def\C{{\mathbb C}}

\def\CC{{\mathbb C}}

\def\NN{{\mathbb N}}

\def\PP{{\mathbb P}}

\newcommand{\ul}{\underline}

\def\cal{\mathcal}

\def\cI{{\cal I}}

\def\piol{{\overline{\pi}}}

\def\piol{{\overline{ \pi}}}

\def\Vol{{\overline{ V}}}

\def\0ol{{\overline{ 0}}}
\def\1ol{{\overline{ 1}}}
\def\2ol{{\overline{ 2}}}
\def\ol2{{\overline{ 2}}}
\def\3ol{{\overline{ 3}}}
\def\4ol{{\overline{ 4}}}
\def\5ol{{\overline{ 5}}}
\def\6ol{{\overline{ 6}}}
\def\7ol{{\overline{ 7}}}
\def\8ol{{\overline{ 8}}}
\def\9ol{{\overline{ 9}}}

\def\P2Skly{\PP^2_{Skly}}

\DeclareMathOperator{\init}{in}

\def\det{\operatorname{det}}

\def\dim{\operatorname{dim}}

\def\max{\operatorname{max}}

\def\min{\operatorname{min}}

\def\rank{\operatorname{rank}}

\def\ul1{\operatorname{\underline{1}}}

\def\l{\leftarrow}

\def\a{\alpha}

\def\l{\lambda}

\def\s{\sigma}
\def\t{\tau}

\def\fS{{\mathfrak S}}

\def\tD{{\tilde D}}

\def\tI{{\tilde I}}

\def\tM{{\tilde M}}

\def\dirlim{\mathop{\vtop{\baselineskip -100pt\lineskip -1pt\lineskiplimit 0pt
\setbox0\hbox{lim}\copy0\hbox to \wd0{\rightarrowfill}}}\limits}
\def\invlim{\mathop{\vtop{\baselineskip -100pt\lineskip -1pt\lineskiplimit 0pt
\setbox0\hbox{lim}\copy0\hbox to \wd0{\leftarrowfill}}}\limits}

\def\I11{{1 \kern -0.8pt \! \mbox{l}}}
\def\mumu{{\mu\kern-4.2pt\mu}}
\def\bfmu{{\mu\kern-4.2pt\mu}}
\def\2slash{\backslash \! \backslash}

\def\boxtimes{\setbox0\hbox{$\Box$}\copy0\kern-\wd0\hbox{$\times$}}

\def\CC{{\mathbb C}} 

\def\D{{\Delta}}
\def\s{{\sigma}}
\def\t{{\tau}}
\def\a{{\alpha}}
\def\fS{{\mathfrak S}}
\def\tD{{\widetilde \Delta}}
\def\tI{{\widetilde I}}
\def\tM{{\widetilde{M}}}
\def\bx{{\mathbf{x}}}
\def\bA{{\mathbf{A}}}
\def\spr{{\rm spr}}

\theoremstyle{definition}
\newtheorem{definition}{Definition}[section]
\newtheorem{example}[definition]{Example}
\theoremstyle{plain}
\newtheorem{theorem}[definition]{Theorem}
\newtheorem{lemma}[definition]{Lemma}

\newtheorem*{corollary*}{Corollary}

\newtheorem{remark}[definition]{Remark}

\usepackage[draft]{fixme}
\fxsetup{theme=color, mode=multiuser, noinline}
\FXRegisterAuthor{rt}{RT}{\color{blue}Rekha}
\FXRegisterAuthor{jk}{JK}{\color{teal}Jack}
\FXRegisterAuthor{td}{TD}{\color{purple}Tim}

\title{Universal Gr\"obner Bases of\\ (Universal) Multiview Ideals}
\author{Timothy Duff, Jack Kendrick, Rekha R. Thomas}
\date{\today}

\begin{document}
\begin{abstract}
    Multiview ideals arise from the geometry of image formation in pinhole cameras, and   universal multiview ideals are their analogs for unknown cameras. We prove that a natural collection of polynomials form a universal Gr\"obner basis for both types of ideals using a criterion introduced by Huang and Larson, and include a proof of their criterion in our setting. Symmetry reduction and induction enable the method to be deployed on an infinite family of ideals. We also give an explicit description of the matroids on which the methodology depends, in the context of multiview ideals. 
\end{abstract}

\maketitle

\section{Introduction}

In this paper, we apply a recent universal Gr\"obner basis criterion due to Huang and Larson~\cite{huanglarson} to show that certain sets of polynomials form universal Gr\"obner bases, for two families of ideals from computer vision. 

First, we recover a result for {\em multiview ideals}, originally~\cite[Theorem 2.1]{aholt2013}.

\begin{theorem}\label{thm:multiview}
For an arrangement of $n$ generic pinhole cameras, the associated 2-,3-, and 4-focals form a universal Gr\"obner basis of the multiview ideal $I_n.$
\end{theorem}

We then extend this result to prove a new result about a universal Gr\"obner basis of {\em universal multiview ideals}, introduced  in~\cite{atlaspaper}.

\begin{theorem}\label{thm:universal-multiview}
    For an arrangement of $n$ unknown cameras, the associated 2-,3-, and 4-focals form a universal Gr\"obner basis of the universal multiview ideal $\tI_n$.
\end{theorem}

We define multiview and universal multiview  ideals in Sections~\ref{sec:multiview ideals} and \ref{sec:universal multiview ideals}. Both form families of ideals indexed by $2 \leq n \in \NN$.
The Huang-Larson criterion for a set of polynomials to form a universal Gr\"obner basis is for a fixed ideal, and it relies on an abstract simplicial complex that is generated by the supports of the candidate polynomials. In order to apply this criterion to a family of ideals, we first need a handle on the corresponding family of complexes. We rely on combinatorics, symmetry reduction, and finally induction, to reduce the checks to a small finite number. We also give explicit descriptions of the complexes as transversal matroids.  

This paper is organized as follows. In Section~\ref{sec:background} we give the necessary background on simplicial complexes and matroids, followed by the Huang-Larson criterion for a set of square-free polynomials to form a universal Gr\"obner basis. In our paper we only need to apply this criterion in the special setting where the underlying varieties are irreducible. In Section~\ref{sec:proof HL}, we include a self-contained proof of the Huang-Larson criterion in this special setting. In Section~\ref{sec:multiview ideals} we apply the criterion to multiview ideals and in Section~\ref{sec:universal multiview ideals} to universal multiview ideals. In the former case we recover a known result (Theorem~\ref{thm:multiview}) using this new technique. Using a similar approach we prove  Theorem~\ref{thm:universal-multiview}, a new result which generalizes~\cite[Theorem 3.2]{atlaspaper}, and resolves the first open question in~\cite[\S 8.1]{atlaspaper}.

\section{Background}
\label{sec:background}

\subsection{Simplicial Complexes and Matroids}
In this subsection, we collect some standard combinatorial terminology and facts that play a role in our results.

An \emph{abstract simplicial complex} $\Delta = (S, \mathcal{I})$ consists of a finite ground set $S$ and a downward-closed family of subsets $\mathcal{I}$: if $A \subseteq S $ satisfies $A \in \mathcal{I},$ then we have $B\in \mathcal{I}$ whenever $B\subseteq A$.
The elements of $\mathcal{I}$ are known as the faces of $\Delta $, and inclusion-maximal faces are called \emph{facets}.
The \emph{dimension} of a  face is one less than its cardinality, and the dimension of a complex is the largest dimension over all of its facets.
A complex is \emph{pure} if all of its facets' dimensions are equal.
Two complexes are \emph{isomorphic} if there is a face-preserving bijection between their ground sets.

Next we recall some basic aspects of {\em Stanley-Reisner theory} ---see eg.~\cite[Ch.~1]{miller-sturmfels,stanley} for more details.
The {\em Stanley-Reisner ideal} of a simplicial complex $\Delta$ on $S$ is the monomial ideal %
$J_\Delta$ in $\CC [x_i : i \in S]$ generated by the square-free monomials $x^A := \Pi_{i \in A} x_i$ such that $A$ is a {\em nonface} of $\Delta$, i.e., $A \not \in \Delta$. 
If $\Delta$ is pure then $J_\Delta$ is equidimensional, i.e., all its associated prime ideals have the same dimension.
On the other hand, every ideal $J$ in $\CC[x_i \,:\, i \in S]$ generated by square-free monomials is the Stanley-Reisner ideal of a simplicial complex $\Delta_J$ on $S$; $x^A \in J$ if and only if $A$ is a nonface of $\Delta_J$. The inclusion-minimal nonfaces of $\Delta_J$ are in bijection with the minimal generators of $J$. The simplicial complex $\Delta_J$ is called the {\em Stanley-Reisner complex} of $J$.

A \emph{matroid} $\mathcal{M} = (S, \mathcal{I})$ on a finite ground set $S$ 
is an abstract simplicial complex, known also as an \emph{independence complex}, which satisfies the \emph{augmentation axiom}: if $I_1, I_2 \in \mathcal{I}$ with $\# I_1 < \# I_2,$ then there exists $s \in I_2 \setminus I_1 $ with $I_1 \cup \{ s \} \in \mathcal{I}.$
Basic matroid theory shows that all independence complexes are pure; their faces and facets are known, respectively, as the \emph{independent sets} and the \emph{bases} of the matroid.
The \emph{rank function} $r_{\mathcal{M}}$ of a matroid $\mathcal{M} = (S, \mathcal{I})$ is the set function satisfying
\[
r_{\mathcal{M}} (X) = 
\displaystyle\max \left\{ \# I
\mid 
I \in \mathcal{I}, \, \, I \subseteq X \right\}
\quad 
\quad 
\forall X \subseteq S.
\]
The rank of $\mathcal{M} = (S , \mathcal{I})$ is defined as $r_{\mathcal{M}} (S)$ and equals the size of any basis of $\mathcal{M}$.

\begin{example}\label{ex:uniform}
The \emph{uniform} matroid $\mathcal{U}_{k,n}$ has ground set $[n] := \{ 1, \ldots , n \}$, and its independent sets are all subsets of $[n]$ of size at most $k.$
\end{example}

\begin{example}\label{ex:transversal}
Let $G $ be a bipartite graph on vertex set $S \sqcup T$.
We define the $G$-\emph{transversal matroid on $S$}, denoted $\mathcal{M} [G, S] = (S , \mathcal{I})$, to consist of all subsets $I \subset S$ that can be covered by a matching in $G.$
\end{example}

In general, we say a matroid is uniform (resp.~transversal) if its is isomorphic to some uniform (resp.~transversal) matroid.
Thus, all uniform matroids are transversal: for the complete bipartite graph $K_{k,n}$, we have $\mathcal{U}_{k,n}\cong\mathcal{M} [K_{k,n}, [n]]$

New matroids can be built from old ones using the \emph{matroid union theorem.}
Given $n$ matroids $\mathcal{M}_1 = (S_1, \mathcal{I}_1), \ldots , \mathcal{M}_n = (S_n, \mathcal{I}_n),$ their \emph{union} 
\[
\mathcal{M} = \mathcal{M}_1 \vee \cdots \vee \mathcal{M}_n = (S , \mathcal{I} )
\]
is defined so that $S = S_1 \cup \cdots \cup S_n$ and
\[
\mathcal{I} = \{ I_1 \cup \cdots \cup I_n \mid I_1 \in \mathcal{I}_1, \ldots , I_n \in \mathcal{I}_n \}.
\]
Note that the ground sets $S_i$ need not be disjoint in the definition of matroid union.
When the ground sets are \emph{pairwise disjoint}, meaning $S_i \cap S_j = \emptyset $ whenever $ i\ne j,$ the matroid union is instead called the \emph{direct sum} and denoted $\mathcal{M}_1 \oplus \cdots \oplus \mathcal{M}_n.$
\begin{theorem}\label{thm:matroid-union}
\cite[Theorem 12.3.1]{oxley} 
The union of matroids $\mathcal{M} = \mathcal{M}_1 \vee \cdots \vee \mathcal{M}_n $ is a matroid, whose rank function $r_{\mathcal{M}}$ can be related to the rank functions $r_1, \ldots , r_n$ of $\mathcal{M}_1, \ldots , \mathcal{M}_n$ as follows:
\[
r_{\mathcal{M}} (X) = \displaystyle\min \left\{  \displaystyle\sum_{i=1}^n r_i (Y) + \# (X \setminus Y) \mid Y \subseteq X
\right\}.
\]
\end{theorem}
\begin{remark}\label{rem:disjoint-bases}
In general, unions of bases need not be bases.
For example, if $\mathcal{M}_1$ and $\mathcal{M}_2$ are the rank-$1$ uniform matroids on $\{ 1\}$ and $\{ 1,2 \}$, respectively, then
\[
B_1 = B_2 = \{ 1 \} 
\quad 
\Rightarrow 
\quad 
B_1 \cup B_2 = \{ 1 \} \subsetneq \{ 1,2 \},
\]
so $B_1\cup B_2$ is not a basis in the union $\mathcal{M}=\mathcal{M}_1 \vee \mathcal{M}_2.$
However, if $B= B_1 \cup \cdots \cup B_n$ with the $B_i$ pairwise disjoint, then $B$ is a basis for $\mathcal{M}=\mathcal{M}_1 \vee \cdots \vee \mathcal{M}_n,$ as this implies $r_{\mathcal{M} } (B) = \# B,$ with the minimum in~Theorem~\ref{thm:matroid-union} attained at any $Y\subseteq B$. 
\end{remark}

\subsection{Universality Criterion for Gr\"obner Bases}

The paper \cite{huanglarson} gives a criterion for determining whether a set of nonzero square-free polynomials is a universal Gr\"obner basis of an ideal. In this section, we give a brief overview of this criterion and describe our general strategy for proving Theorems \ref{thm:multiview} and \ref{thm:universal-multiview}.

Fix a variety $V\subset \C^N$. For any polynomial $f\in \C[x_1, \ldots, x_N]$, define the {\em spread} of $f$ to be the collection of $i\in [N]$ such that $x_i$ is in the support of $f.$ For a collection of nonzero polynomials $f_1, \ldots, f_r$ let $\D(f_1, \ldots, f_r)$ be the simplicial complex on $[N]$ whose nonfaces are generated by the spreads of $f_1, \ldots, f_r.$ 

There is a natural embedding of $V$ in the multiprojective space $(\PP^1)^N$ obtained by homogenizing each coordinate $x_i$, i.e., sending $x_i \mapsto [x_i:1]$. Let $\Vol$ denote the closure of $V$ inside $(\PP^1)^N.$ Given a subset $U\subset [N],$ consider the projection
\begin{align}
\piol_U:(\PP^1)^N\to(\PP^1)^U \label{eq:Uprojection}\\
\left( [x_1 : x_{N+1}], \ldots , [x_N : x_{2N}] \right) &\mapsto \left( [x_i : x_{i+N} ] \mid i \in U \right). \nonumber 
\end{align}
The following theorem gives a sufficient condition for when $\{ f_1, \ldots , f_r \}$ forms a universal Gr\"{o}bner basis for the vanishing ideal $\cI(V)$.

\begin{theorem}[Theorem 2.7 in \cite{huanglarson}]\label{thm:huang-larson}
    Let $V \subset \CC^N$ be a closed subvariety, and $f_1, \ldots, f_r \in \mathcal{I} (V) \setminus \{ 0 \} $ be square-free  polynomials.
    If $\piol_U(\Vol) = (\PP^1)^U$ for every facet $U$ of $\Delta(f_1, \ldots, f_r)$, then $\{f_1, \ldots, f_r\}$ is a universal Gr\"obner basis of $\mathcal{I} (V).$ 
\end{theorem}

To apply Theorem~\ref{thm:huang-larson}, we must show that the projection $\piol_U:\Vol\to(\PP^1)^U$ is surjective
for each facet of the simplicial complex $\D(f_1, \ldots, f_r).$ Since our varieties are affine, it will be convenient to instead consider the projections $\pi_U:V\to(\CC^1)^U.$ The following simple lemma expresses the the surjectivity, and hence dominance, of $\piol_U$ in terms of the original affine variety $V$, when $V$ is irreducible.

    \begin{lemma}\label{lem:project-to-C}
    Let $V \subset (\CC^1)^N $ be irreducible. 
    For any subset $U\subset \{ 1, \ldots , N \}$, the projection $\pi_U : V \to (\CC^1)^U$ is dominant if and only if $\piol_U : \overline{V} \to (\PP^1)^U$ is surjective. 
    \end{lemma}
    \begin{proof}
        Consider the following diagram:
        $$
        \begin{array}{ccc}
             V & \xhookrightarrow{i_V} & \overline{V} \\
             &&\\
             {\tiny{\pi_U}}{\downarrow} &  & \downarrow {\tiny{\piol_U}}\\
             &&\\
             (\CC^1)^U & \xhookrightarrow{i_U} & (\PP^1)^U
        \end{array}
        $$
        Since the horizontal inclusion maps are birational isomorphisms, the image of $\pi_U$ is dense if and only if the image of $\piol_U$ is. 
        Furthermore, dominance and surjectivity of $\piol_U$ are equivalent by the projective elimination theorem.
    \end{proof}

While~Theorem~\ref{thm:huang-larson} gives a sufficient condition for a collection of polynomials to form a universal Gr\"obner basis of a single ideal, in this paper we wish to prove that a given sequence of sets of polynomials form universal Gr\"obner bases for a nested sequence of ideals $I_n$, indexed by $n \in \NN$. Our general strategy is as follows:
\begin{enumerate}
    \item For each $n$, describe the simplicial complex $\D_n$ needed in Theorem~\ref{thm:huang-larson} for $I_n$, 
    and determine a ``growth rule" as $n$ increases.
    \item Use symmetry to reduce the number of facets that need to be checked in Theorem~\ref{thm:huang-larson}.
    \item Prove a base case via direct computation in Macaulay2~\cite{M2}. 
    \item Then apply induction to prove  the general case.
\end{enumerate}

\section{A proof of~Theorem~\ref{thm:huang-larson} in the irreducible case} \label{sec:proof HL}

In this section, we provide a short proof of~Theorem~\ref{thm:huang-larson} under the assumption that the variety $V$ is irreducible. 
This gives us a self-contained route towards our main results, where this assumption always holds.

Our task is to show that under the hypothesis of Theorem~\ref{thm:huang-larson}, $\{f_1, \ldots, f_r\}$ is a Gr\"obner basis for $\mathcal{I}(V)$ with respect to any monomial order $<$ on $\CC[x_1, \ldots, x_N]$.

Let $\CC[x_1, \ldots, x_N, x_{N+1}, \ldots, x_{2N}]$ be the coordinate ring of $(\PP^1)^N$ where variable $x_i$ is paired with $x_{N+i}$ for $i=1, \ldots, N$. The {\em multi-homogenization} of a polynomial $f \in \CC[x_1, \ldots, x_N]$ is 
$f^h \in \CC [x_1, \ldots ,x_{2N}]$ defined as 
\begin{align}\label{eq:multihomogenization}
f^h(x_1, \ldots , x_N ) = 
x_{N+1}^{\deg_{x_1} (f)} \cdots x_{2N}^{\deg_{x_N} (f)} \cdot 
f (x_1/x_{N+1}, \ldots , x_N / x_{2N}).
\end{align}
The vanishing ideal of $\overline{V}$ in $\CC[x_1,\ldots,x_{2N}]$ is 
\begin{equation}\label{eq:multihomogenized-ideal}
\mathcal{I} (\overline{V}) = \langle f^h \mid f \in \mathcal{I} (V) \rangle.
\end{equation}
In order to show that $\{f_1, \ldots, f_r\}$ is a Gr\"obner basis of $\mathcal{I}(V)$ with respect to any monomial order on $\CC[x_1, \ldots, x_N]$, it suffices to argue that $\{f_1^h, 
\ldots, f_r^h\}$ is a Gr\"obner basis for $\mathcal{I}(\overline{V})$ with respect to any product order such that $x_{j+N} < x_i$ for all $i,j.$ These orders can be viewed as natural extensions of monomial orders on $\CC[x_1, \ldots, x_N]$ to monomial orders on $\CC[x_1, \ldots, x_{2N}].$ If $f_1^h, \ldots, f_r^h$ form a Gr\"obner basis of $\cI(\Vol)$ with respect to $<,$ then $f_1, \ldots, f_r$ form a Gr\"obner basis of $\cI(V)$ with respect to the corresponding monomial order on $\CC[x_1, \ldots, x_N],$ see for example, \cite[Exercise 15.39]{eisenbud}.

Fix such an order $<$ on $\CC[x_1, \ldots, x_{2N}]$ and 
define the monomial ideal %
\begin{equation}\label{eq:stanley-reisner-ideal}
J := \langle 
\init_< (f_1^h) , \ldots , 
\init_< (f_r^h)
\rangle \subset \CC [x_1, \ldots, x_{2N}].
\end{equation}
To show that $f_1^h, \ldots, f_r^h$ form a Gr\"obner basis of $\cI(\Vol)$ with respect to $<,$ 
we will apply the following result of Conner, Han, and Michalek~\cite{conner2023sullivant} to $I = \mathcal{I} (\overline{V}).$

\begin{lemma}[Lemma 3.1 in \cite{conner2023sullivant}]
    Suppose $I$ is a homogeneous prime ideal and $g_1, \ldots, g_r$ are elements of $I.$ Assume the following hold for the monomial order $<:$
    \begin{enumerate}
    \item $\init_< (g_1), \ldots , \init_< (g_r)$ are square-free,
    \item $J:=\langle\init_<(g_1), \ldots, \init_<(g_r)\rangle$ is equidimensional, and 
    \item $\dim I = \dim J$, and $\deg I  = \deg J$.
    Then, $g_1, \ldots, g_r$ is a Gr\"obner basis of $I$ with respect to $<.$
\end{enumerate}
\end{lemma}

Since the variety $V$ is irreducible by assumption , so is $\Vol$, and thus the vanishing ideal $\cI(\Vol)$ is prime, and also homogenous. By assumption (in Theorem~\ref{thm:huang-larson}), the initial monomials $\init_<(f_1^h), \ldots, \init_<(f_r^h)$ are square-free and so we need only prove that conditions (2) and (3) hold. %
Let $\spr f$ denote the spread of a polynomial $f$.

The following lemma, summarizing two results in \cite{huanglarson}, will be useful for us.  
\begin{lemma} \label{lem:nothing supported on U}
    Suppose $V\subset \C^N$ is a closed subvariety. For a subset $U\subset[N],$ we have that $\piol_U(\Vol)\neq (\PP^1)^U$ if and only if there is a nonzero $f\in \mathcal{I}(V)$ such that 
    $\spr f\subset U.$ Alternatively, $\piol_U(\Vol) = (\PP^1)^U$ if and only if $\mathcal{I}(V)\cap\C[x_i:i\in U] = 0.$
\end{lemma}
\begin{proof}
    This is an easy consequence of Lemmas 2.1 and 2.2 in \cite{huanglarson}.
\end{proof}

Since the variety $V$ is irreducible, the sets $U\subset [N]$ with $\piol_U(\Vol) = (\PP^1)^U$ 
are precisely the independent sets in an \emph{algebraic matroid}~\cite[\S 6.7]{oxley}, whose rank is $\dim V.$
The bases in this matroid correspond to transcendence bases of rational functions in the set $\{ x_1 , \ldots , x_N \} \subset \CC (V)$.
As in~\cite{rosen2020algebraic}, we call this the \emph{algebraic matroid of} $V.$ 

\begin{lemma} \label{lem:alg matroid}
    Under the hypotheses of Theorem~\ref{thm:huang-larson} (and since $V$ is irreducible), the simplicial complex $\D(f_1, \ldots, f_r)$ is the algebraic matroid of $V.$ 
    In particular, $\D(f_1, \ldots, f_r)$ is a pure complex of dimension $\dim V - 1.$
\end{lemma}

\begin{proof}
    Under the hypotheses of~Theorem~\ref{thm:huang-larson}, every face of $\Delta (f_1,\ldots , f_r)$ is independent in the algebraic matroid of $V$.
    On the other hand, if $U$ is a nonface of $\Delta (f_1,\ldots , f_r)$, then $U$ contains the spread of some $f_i,$ and~Lemma~\ref{lem:nothing supported on U} implies $U$ is dependent.
    Thus, the algebraic matroid of $V$ coincides with the simplicial complex $\Delta (f_1,\ldots , f_r)$. Furthermore, since the bases of a matroid all have the same size, $\D(f_1, \ldots, f_r)$ is a pure simplicial complex.
\end{proof}

Returning to the proof of~Theorem~\ref{thm:huang-larson}, we now prove that condition (2) holds, i.e., that $J = \langle \init_< (f_1^h) , \ldots , \init_< (f_r^h)\rangle$ is equidimensional, using Stanley-Reisner theory. Observe that the Stanley-Reisner complex $\Delta_J$ is closely related to the simplicial complex $\Delta (f_1, \ldots , f_r)$; for any $f \in \CC[x_1, \ldots, x_N]$ there is a bijective correspondence between $\spr f$ in $[N]$ and the support of $\init_<(f^h)$ in $[2N]$. Hence, the nonfaces of $\D_J$ and $\Delta(f_1, \ldots , f_r)$ are in bijection. 
Moreover, by~\cite[Lemma 2.10]{huanglarson}, there is a bijection sending facets $U\in \Delta (f_1, \ldots , f_r)$ of size $k$ to facets $\widetilde{U} \in \Delta_J$ of size $k+N.$ %
    
\begin{example}
For example, if $f=x_1 + x_2x_3 \in \C[x_1,x_2,x_3]$, then $\spr f = \{1,2,3\}$ and $\Delta_{\langle f \rangle}$ has the single generating nonface $\{1,2,3\}$. Double the variables by letting $x_1 \leftrightarrow x_4, x_2 \leftrightarrow x_5, x_3 \leftrightarrow x_6$.
Then $f^h = x_1 x_{5} x_6 + x_4 x_2 x_3$, and for any term order with $\init_< (f^h ) = x_1 x_5 x_6$, the complex $\Delta_{\langle \init_< (f^h)\rangle}$ is generated by the nonface $\{ 1, 5, 6 \}$. The indices $2,3$ in $\{1,2,3\}$ have been replaced by $5,6$ in $\{1,5,6\}$.
We also have the following bijective correspondence between facets:
\begin{align*}
F =\{ 1,2 \} \in \Delta_{\langle f \rangle}
&\leftrightarrow 
\widetilde{F} = \{ 1,2,3,4,5 \} \in \Delta_{\langle \init_< (f)\rangle},\\
F =\{ 1,3 \} \in \Delta_{\langle f \rangle}
&\leftrightarrow 
\widetilde{F} = \{ 1,2,3,4,6\} \in \Delta_{ \langle \init_< (f)\rangle},\\
F =\{ 2,3 \} \in \Delta_{\langle f \rangle}
&\leftrightarrow 
\widetilde{F} = \{ 2,3,4,5 , 6\}\in \Delta_{ \langle \init_< (f) \rangle}.
\end{align*}
\end{example}

\begin{lemma}
    The ideal $J$ is equidimensional.
\end{lemma}
\begin{proof}
As observed in the previous discussion, the ideal $J$ is the Stanley-Reisner ideal of the simplicial complex $\D_J = \D_{\langle \init_< (f_1^h) , \ldots , \init_< (f_r^h)\rangle}.$ Recall that the Stanley-Reisner ideal of a pure simplicial complex is equidimensional and so it suffices to prove that $\D_J$ is pure. 

There exists a bijection mapping each facet $U\in\D_J$ of size $k$ to a facet $\widetilde{U}\in \D_J$ of size $k+N.$ Therefore, since  $\D(f_1, \ldots, f_r)$ is pure, so is $\D_J$.
\end{proof}

We now establish condition (3). 
First we show the  equality $\dim \cI(\Vol) = \dim J.$

\begin{lemma}\label{lem:dim-I=dim-J}
    The equality holds:
    $$\dim \cI(\Vol) = \dim J$$
\end{lemma}

\begin{proof}
Note that  $\dim \cI(\Vol)$ is given in terms of the affine variety $V$ by the equalities
\begin{equation}\label{eq:dim-variety-matroid}
\dim \mathcal{I} (\overline{V}) = 
N + \dim \mathcal{I} (V)
= 
N + \dim (V).
\end{equation}

Since $J$ is the Stanley-Reisner ideal of  $\D_{\langle \init_< (f_1^h),\ldots,\init_< (f_r^h)\rangle},$ we have that
\begin{align*}
    \dim J &= \dim \Delta_{\langle \init_< (f_1^h),\ldots,\init_< (f_r^h)\rangle} + 1\\
&= N + \dim \Delta (f_1, \ldots , f_r) + 1 \tag{\cite[Lemma 2.10]{huanglarson}}\\
&= N + \rank \Delta (f_1, \ldots , f_r) \\
&= N + \dim (V) \tag{by~Lemma~\ref{lem:alg matroid}}\\
&= \dim \mathcal{I} (\overline{V}) \tag{by~\eqref{eq:dim-variety-matroid}} 
\end{align*}
This completes the proof.
\end{proof}

Finally, we show that $\deg \cI(\Vol) = \deg J$. 

\begin{lemma}
    The equality holds:
    $$\deg \cI(\Vol) = \deg J$$
\end{lemma}
\begin{proof}
The variety $\mathcal{V} (J)\subset (\PP^1)^N$ is a union of linear subspaces indexed by facets,
\[
\mathcal{V} (J) = \displaystyle\bigcup_{\substack{F \in \Delta_{\langle \init_< (f_1^h) , \ldots , 
  \init_< (f_r^h)\rangle}\\F \text{ a facet}}} L_F.
\]
Since $J$ is generated by square-free monomials, it is a  radical ideal and defines a subvariety of $(\PP^1)^{2N-1}$, of degree
\[
\deg (J) = \# \text{facets of } \Delta_J = \# \mathcal{B},
\]
where $\mathcal{B}$ denotes the set of bases in the algebraic matroid of $V.$

Let $I:= \cI(\Vol)$ and $\init_< (I) = \langle \init_< (f^h ) \mid f \in I \rangle$ be the initial ideal of $\cI(\Vol).$ Note that the $\deg  \init_< (I) = \deg I$ and $\dim \init_< (I)  = \dim I= \dim J$ by Lemma~\ref{lem:dim-I=dim-J}. Then, since $J\subseteq \init_< (I),$ 
we have that
\[
\deg I = \deg \init_< (I) \le \deg J = 
\# \mathcal{B}
\le \displaystyle\sum_{U\in \mathcal{B}} \# \piol_U^{-1} \left( p_U \right),
\]
where each $p_U \in (\PP^1)^U$ is a generic point in the codomain of the projection $\pi_U.$
Each term $\# \piol_U^{-1} \left( p_U \right)$ in the sum above is a \emph{multidegree} of the variety $\overline{V} = \mathcal{V} ( \mathcal{I} (\overline{V}))$.
A result of van der Waerden~\cite{vdw} expresses the standard degree $\deg \mathcal{I} (\overline{V})$ as the sum of multidegrees,
\[
\displaystyle\sum_{U\in \mathcal{B}} \# \piol_U^{-1} \left( p_U \right)
=
 \deg (\mathcal{I} (\overline{V})).
\]
(See also~\cite[Proposition 1.7.3]{groebner-geom}, for a more general statement.)
Thus, we have
\[
\deg I \le 
\deg J \le 
\displaystyle\sum_{U\in \mathcal{B}} \# \piol_U^{-1} \left( p_U \right)
=   \deg I
\]
and so $\deg \cI(\Vol) = \deg J$ as desired
\end{proof}

Thus condition (3) is also satisfied, and  $f_1^h, \ldots , f_r^h$ form a Gr\"{o}bner basis for $\mathcal{I} (\overline{V})$ with respect to $<.$

\section{Multiview Ideals}\label{sec:multiview ideals}

 In this section, we define the multiview ideal and the family of polynomials known as %
 $k-$focals, which will form a universal Gr\"obner basis of the multiview ideal. We begin with a brief overview of the necessary background from the literature on computer vision. For further details, we refer the reader to \cite{atlaspaper, aholt2013}.

The standard model of a pinhole camera is a surjective linear map  $\PP^3\dashrightarrow\PP^2$ which is represented up to scale by a $3\times 4$ matrix $A$ of full rank. The camera maps a {\em world point} $q\in \PP^3$ to its {\em image point} $Aq\in\PP^2.$ 

A configuration of $n$ cameras,  represented by the matrices $A_1, \ldots, A_n,$ is {\em generic} if all $4\times 4$ minors of the matrix $\begin{bmatrix}
    A_1^\T & \ldots & A_n^\T
\end{bmatrix}$ are nonzero. We fix a generic configuration of cameras $(A_1, \ldots, A_n).$ The {\em multiview variety} of $(A_1, \ldots, A_n)$ is 
the closure in $\left( \PP^2 \right)^n$ of the image of the rational map,
\begin{align*}
\PP^3 &\dashrightarrow \left( \PP^2 \right)^n \\
q &\mapsto (A_1 q , \ldots , A_n q).
\end{align*}
This is an irreducible variety containing the $n$-tuples of  images of world points $q\in\PP^3.$  
Its vanishing ideal is the multiview ideal associated to cameras $(A_1, \ldots, A_n)$. Equivalently, this is the vanishing ideal of the affine cone over the multiview variety. Although the multiview ideals for $n$ generic cameras will vary with the cameras, the combinatorics governing such ideals will depend only on $n$---thus, abusing notation, we let $I_n$ denote the multiview ideal associated to $n$ generic cameras. %
To mirror the setting of \cite{huanglarson}, we work primarily with the affine cone over the multiview variety, which we denote by $M_n.$

For a subset $\sigma\subset[n]$ of size $k,$ and $x_i = (x_{i1},x_{i2},x_{i3})^\top$, a $k-$focal is a maximal minor of the $3k \times (4+k)$ matrix
\begin{equation} \label{eq:kfocal matrix}
    \begin{bmatrix}
        A_{\s_1} & x_1 & 0 & \cdots & 0 \\
        A_{\s_2} & 0 & x_2 & \cdots & 0 \\
        \vdots & \vdots & \vdots & \ddots & \vdots \\
        A_{\s_k} & 0 & 0 & \cdots & x_k
    \end{bmatrix}.
\end{equation}
The work \cite{aholt2013} proves that the set of all $2-,3-,$ and $4-$focals forms a universal Gr\"obner basis of $I_n.$ We recover this result using Theorem~\ref{thm:huang-larson}.

Note that, since $3k\geq 4+k$ for $k\geq 2,$ a maximal minor is determined by a choice of $4+k$ rows of the above matrix. Thus, each $k-$focal is uniquely determined by a choice of $k$ cameras and a choice of 
$4+k$ coordinates from the $3$-tuples $x_1, \ldots, x_k$. It will be convenient to consider a coarse classification of focals which we call the {\em profile}. For any polynomial $f\in \C[x_1, \ldots, x_n],$ the profile of $f$ is an $n-$dimensional vector $\type(f)$ with each coordinate corresponding to the index of a camera and each entry being the number of coordinates appearing as a variable in $f$, from the corresponding $x_i$.  For example, consider the 2-focal
\begin{equation*}
    f = \det \begin{bmatrix}
        A_1 & x_1 & 0 \\
        A_2 & 0 & x_2
    \end{bmatrix}.
\end{equation*}
We see that $\type(f)=(3,3,0\ldots, 0)$ all three coordinates of $x_1$ and $x_2$ appear as variables in the determinant. Note that if $f$ is a $k-$focal and any entry of $\type(f)$ is equal to 1, then $f$ is a monomial multiple of some focal of lower degree. For example, consider the 3-focal $g$ given by the minor
\begin{equation*}
    g =\det \begin{bmatrix}
        A_1 & x_1 & 0 & 0\\
        A_2 & 0 & x_2 & 0\\
        A_{31} & 0 & 0 & x_{31}
    \end{bmatrix}.
\end{equation*}
Then $\type(g)=(3,3,1,0\ldots, 0)$ and by Laplace expansion we see that $g$ is $x_{31}\cdot f.$ Since any monomial multiple of a focal is a redundant generator of the multiview ideal, we only consider focals $f$ where $\type(f)$ has no entry equal to 1. 

\subsection{The simplicial complex}\label{sec:mv-complex} For $n\geq 4,$ define $\D_n$ to be the simplicial complex whose nonfaces are generated by the spreads of the $2-,3-, 4-$focals of a generic camera configuration $(A_1, \ldots, A_n).$ We determine the dimension and number of facets of $\D_n$ and show that each facet has one of two profiles up to permutation of cameras.

Note that each face $U\in\D_n$ corresponds to the support of some square-free monomial $x^U = \prod_{ij\in U}x_{ij}\in\C[x_1, \ldots, x_n].$ Thus, for $U\in\D_n$ we define its profile as $\type(U) = \type(x^U).$ The following results show that if $U\in\D_n$ is a facet, it has one of two profiles up to permutation of cameras.

First, consider the case where $n=4.$ Up to permutation of cameras, any $2-,3-,$ or $4-$focal has profile $(3,3,0,0),(3,2,2,0),$ or $(2,2,2,2)$. To each $k-$focal $f$ we associate a {\em spread monomial}: the product of all variable $x_{ij}$ in the spread of $f.$ The facets of the simplicial complex $\D_4$ correspond to all square-free monomials in the $12$ variables 
$$x_{11},x_{12},x_{13},x_{21},x_{22},x_{23},x_{31},x_{32},x_{33},x_{41},x_{42},x_{43}$$
with maximal support that are not divisible by any spread monomial. For example, one such monomial is 
$$x_{11}x_{12}x_{13}x_{21}x_{22}x_{31}x_{41}$$
which has profile $(3,2,1,1)$. %
Another possibility is the monomial
$$x_{11}x_{12}x_{21}x_{22}x_{31}x_{32}x_{41}$$
which has profile $(2,2,2,1)$. We show that up to permutation, any facet of $\D_4$ has profile $(3,2,1,1)$ or $(2,2,2,1).$

\begin{lemma}\label{lem:Delta_4} The complex $\Delta_4$ is pure, 6-dimensional and has 648 facets whose profiles are obtained from all permutations of $(3,2,1,1)$ and $(2,2,2,1)$, and choices of image variables from each camera.
\end{lemma}

\begin{proof}
    Recall that the Stanley-Reisner ideal of $\D_4$ is generated by all square-free monomials  $\prod_{ij\in W}x_{ij}$ where $W$ is a {\em nonface} of $\D_4.$ The nonfaces of $\D_4$ are generated by the spreads of all $2-,3-,$ and $4-$focals and each of these focals has profile $(3,3,0,0), (3,2,2,0)$ or $(2,2,2,2)$ up to permutation of cameras. Therefore, the Stanley-Reisner ideal of $\D_4$ is generated by monomials with profile $(3,3,0,0), (3,2,2,0)$ or $(2,2,2,2)$ up to permutation. Moreover, any focal is uniquely determined by a choice of $k$ cameras and $4+k$  coordinates $x_{ij}$ and so for any permutation and choice of $\bx$ coordinates that could lead to a square-free monomial with profile $(3,3,0,0), (3,2,2,0)$ or $(2,2,2,2),$ it is simple to construct a corresponding $k-$focal. Thus, the Stanley-Reisner ideal of $\D_4$ is generated by all square-free monomials with profile $(3,3,0,0), (3,2,2,0)$ or $(2,2,2,2)$ up to permutation.

    The profile of each facet $U$ of $\D_4$ is a $4$-dimensional vector $\type(U)$ with entries in $\{0, 1, 2, 3\}.$ Up to permutation of cameras, $\type(U)$ may be written in lexicographic order. Since the monomial corresponding to $U$ must not be divisible by any monomial in the Stanley-Reisner ideal, $\type(U)$ can have at most one entry equal to 3. Once the first entry is chosen to be 3, $\type(U)$ must be the largest vector below $(3, 2, 2, 0)$ in lexicographic order.  It follows that $\type(U) = (3, 2, 1, 1).$ If the first entry of $\type(U)$ is 2, then $\type(U)$ is the largest vector below $(2,2,2,2)$ in the lexicographic order and so $\type(U) = (2,2,2,1).$ It follows that any facet of $\D_4$ has profile $(3,2,1,1)$ or $(2,2,2,1).$ On the other hand, any square-free monomial with profile $(3,2,1,1)$ or $(2,2,2,1)$ is not divisible by any monomial in the Stanley-Reisner ideal of $\D_4$ and thus corresponds to a facet of the complex. Thus, the facets of $\D_4$ are in bijective correspondence with the set of all monomials with profile $(3,2,1,1)$ of $(2,2,2,1)$ up to permutation. In particular, $\D_4$ is a pure 6-dimensional simplicial complex.

    By the above discussion, all facets of $\D_4$ have profile $(3,2,1,1)$ or $(2,2,2,1).$ For facets with profile $(3,2,1,1)$ there are $4$ ways to choose the camera that contributes $3$ variables and then $3$ ways to choose the one that contributes $2$ variables which is a total of $12$. The cameras that contribute $2$ or $1$ variables have $3^3$ ways to pick the variables. Therefore, there is a total of $12 \times 27 = 324$ such facets. By a similar argument, for the facet $(2,2,2,1)$, there are $4$ ways to choose the camera that contributes $1$ variable and $3^4$ ways for the various cameras to choose the variables they want to contribute, making another $324$ possibilities. Together, we have $648$ facets, each of dimension $6$.%
\end{proof}

We now consider the general case $n\geq 4.$ Again, we show that up to permutation of cameras the facets of $\D_n$ have  profile $(3,2,1,1,\ldots, 1)$ or $(2,2,2,1,\ldots, 1).$

\begin{theorem}\label{thm:facet-types}
    The simplicial complex $\Delta_n$ for $n \geq 4$ is pure, $(n+2)-$dimensional and has only two profiles of facets up to permutation of cameras and choice of variables in each camera plane. There are $n(n-1) 3^{n-1}$ facets with profile $(3,2,1,1,\ldots, 1)$ and ${n \choose 3} 3^n$ facets with profile $(2,2,2,1,\ldots, 1)$. 
\end{theorem}

\begin{proof}

    Following the same argument as in Lemma~\ref{lem:Delta_4}, the Stanley-Reisner ideal of $\D_n$ is generated by all square-free monomials with profile one of $$(3,3,0\ldots, 0), (3,2,2,0,\ldots, 0)
    \textup{ or } (2,2,2,2,0,\ldots, 0)$$ 
    up to permutation of cameras.

    Choose a facet $U\in\D_n.$
    As in Lemma~\ref{lem:Delta_4}, $U$ corresponds to a maxmimally supported square-free monomial that is not divisible by any monomial in the Stanley-Reisner ideal. Up to permutation of cameras, we may assume that $\type(U)$ is in lexicographic order. In particular, the first entry of $\type(U)$ is either 2 or 3.

    The non-faces of $\D_n$ are generated by the spreads of all $2-,3-, 4-$focals and so, as in Lemma \ref{lem:Delta_4}, the Stanley-Reisner ideal of $\D_n$ is generated by all monomials corresponding to vectors of the form $(3,3,0,0, 0,\ldots), (3,2,2,0, 0,\ldots)$ or $(2,2,2,2,0\ldots)$ up to permutation of cameras. If the first entry of $\type(U)$ is 3, then $\type(U) = (3, 2, 1, 1, \ldots, 1)$ since this is the maximal element below $(3, 2, 2, 0, \ldots, 0)$ in the lexicographic order. If the first entry of $\type(U)$ is 2, then $\type(U) = (2,2,2,1,1\ldots, 1)$ since this is the maximal element below $(2,2,2,2,0,\ldots,0)$ in the lexicographic order. Thus, every facet of $\D_n$ has profile $(3,2,1,1,\ldots, 1)$ or $(2,2,2,1,\ldots, 1).$ Note that the sum of entries in each of these vectors is $n+3$ and so the simplicial complex $\D_n$ is pure and $(n+2)-$dimensional. Moreover, any square-free monomial with profile $(3,2,1,1\ldots, 1)$ or $(2,2,2,1,\ldots, 1)$ up to permutation of cameras is not divisible by any element of the Stanley-Reisner ideal. It follows that the facets of $\D_n$ are in bijective correspondence with all square-free monomials with profile $(3,2,1,1\ldots, 1)$ or $(2,2,2,1,\ldots, 1)$ up to permutation.

    We now count the number of facets. For facets with profile $(3,2,1,1,\ldots, 1),$ there are $n$ choices of camera contributing 3 variables and $(n-1)$ choices of camera contributing 2 variables, and the remaining cameras each contribute 1 variable. There are 3 choices of variables for each camera contributing either 1 or 2 variables and so in total there are $n(n-1)3^{n-1}$ facets with profile $(3,2,1, 1,\ldots, 1).$ For facets with profile $(2,2,2,1,\ldots, 1),$ there are ${n\choose 3}$ choices of cameras contributing 2 variables and the remaining cameras each contribute one variable. For each of the $n$ cameras there are 3 choices of variables and so in total there are ${n\choose 3}3^n$ facets with profile $(2,2,2,1,\ldots, 1).$
\end{proof}

The simplicial complex $\D_n$ is in fact the independence complex of a matroid, a result that  will follow from~Lemma~\ref{lem:alg matroid} and  the next subsection. Our next result answers the question, \emph{what kind of matroid?}

\begin{theorem}\label{thm:matroid-multiview}
The simplicial complex $\Delta_n$ is a  transversal matroid of rank $n+3$, isomorphic to the union of the uniform matroid $\mathcal{U}_{3,3n}$ on all $x$-variables and the direct sum $\mathcal{U}_{1,n}^{\oplus n}$ over the $n$ subsets $\{ x_{i1}, x_{i2}, x_{i3} \}$.
\end{theorem}

\begin{proof}
Let $U$ be any facet of $\Delta_n$. 
Whether $U$ has profile 
$(3,2,1,1,\ldots,1)$ or $(2,2,2,1,\ldots,1)$, it must must contain at least one element from  $\{ x_{i1}, x_{i2}, x_{i3} \}$ for each $i=1,\ldots,n$.
In other words, $U$ must contain a basis $B$ of the direct sum $\mathcal{U}_{1,n}^{\oplus n}$.
Now, since the remaining 3 elements of $U \setminus B$ form a basis of $\mathcal{U}_{3,3n}$, we have that $U$ is a basis in $\mathcal{U}_{3,3n}\vee \mathcal{U}_{1,n}^{\oplus n}.$
A similar proof shows that any face of $\Delta_n$ is a union of independent sets in $\mathcal{U}_{3,3n}$ and $\mathcal{U}_{1,n}^{\oplus n}.$
Since the class of transversal matroids is closed under union (\cite[Corollary 12.3.8]{oxley}), this completes the proof.
\end{proof}

The structure of $\Delta_n$ as a transversal matroid can be seen by explicitly constructing a bipartite graph $G = (S \sqcup T , E)$ with $\Delta_n \cong \mathcal{M} [G, S]$, according to the following rule: set vertices $S=\{ x_{11}, \ldots , x_{n3} \},$ $T=\{ a,b,c \} \sqcup \{ v_1, \ldots , v_n \} $, and edges
\[
E = \displaystyle\bigcup_{\substack{1\le i \le n \\ 1\le j \le 3}}\{
x_{ij} v_i,  
x_{ij} a , 
x_{ij} b ,
x_{ij} c
\} .
\]
In matroid theory, it is well-known that every transversal matroid is representable over a sufficiently large (finite) field (cf.~\cite[Ch.~6]{oxley}.)
On the other hand, the matroids $\Delta_n$ are not realizable over arbitrary fields. 
For example, starting from $\Delta_4$, we may delete the variables $x_{21}, x_{31}, x_{41}$ and contract the variables $x_{11},x_{12},x_{22},x_{32}, x_{42}$ to obtain the uniform matroid $\mathcal{U}_{2,4}$ on the remaining variables $x_{13}, \ldots , x_{43}$ as a minor.
This implies that $\Delta_n$ for $n\ge 4$ is not representable over the field with two elements.
In particular, the matroid $\Delta_n$ is non-graphic~\cite[Chapter 5]{oxley}.

\subsection{Applying the Huang-Larson theorem}

Let $M_n \subset \CC^{3n}$ denote the affine cone over the multiview variety from $n$ generic cameras and let $I_n$ denote its vanishing ideal. By Lemma \ref{lem:project-to-C}, it is sufficient to prove that for each facet $U\in\D_n$ the projection $\pi_U:M_n\to (\CC^1)^U$ is dominant. 

First, we prove that it suffices to check that the projection is dominant for a small number of facets.     The structure of the simplicial complex $\D_n$ is highly symmetric. There is a natural permutation action of the symmetric group $\fS_n$ on both the simplicial complex and the polynomial ring $\CC[x_1, \ldots, x_n]$ that corresponds to the permutation of cameras. For each permutation $\t\in \fS_n$ and face $U = \{x_{ij}\}$ of $\D_n,$ we write $\t U$ to denote the face $\{x_{\t(i)j}\}$ of $\D_n.$ Note that by Theorem \ref{thm:facet-types}, if $U$ is a facet of $\D_n$ then so is $\t U$ for each permutation $\t\in\fS_n$ and any facet $W\in \D_n$ can be written as $\t U$ for some $\t\in\fS_n$ and some facet $U$ with profile $(3,2,1,1,\ldots, 1)$ or $(2,2,2,1,\ldots, 1).$ 

    The symmetric group $\fS_n$ acts on the polynomial ring $\CC[x_1, \ldots, x_n]$ in the same way and we see that the $k-$focals of the camera configuration $(A_1,\ldots, A_n)$ exhibit similar symmetries. For each polynomial $f \in \CC[x_1, \ldots, x_n],$ we write $\t f$ to denote the polynomial achieved by replacing $x_{ij}$ with $x_{\t(i)j}.$ Note that if $f$ is a $k-$focal arising as a maximal minor from some choice of subset $\s\subset [n],$ then $\t f$ is a $k-$focal arising as a maximal minor from the subset $\t(\s)\subset[n].$
    
    We now use these symmetries to reduce the number of computations needed to prove that $2-,3-,4-$focals form a universal Gr\"obner basis of the multiview ideal.

    \begin{lemma}\label{lem:symmetry-reduction}
    It suffices to prove that the map~\eqref{eq:Uprojection} is dominant for facets with profile $(3,2,1,1,\ldots, 1)$ and $(2,2,2,1,\ldots, 1)$.
    \end{lemma}

    \begin{proof}
        By Theorem \ref{thm:facet-types}, each facet of $\D_n$ corresponds to some permutation of either $(3, 2,1 ,1\ldots, 1)$ or $(2,2,2,1,\ldots, 1).$ Thus, if $W$ is a facet of $\D_n$ there exists a facet $U$ with profile $(3,2,1,1,\ldots, 1)$ or $(2,2,2,1\ldots, 1)$ and a permutation $\t\in\fS_n$ such that $W = \t U.$ 

        By Lemma~\ref{lem:project-to-C} it suffices to work with 
        $\pi_U$. 
        Suppose that the projection $\pi_U$ is dominant for all facets $U$ with profile $(3,2, 1, 1,\ldots, 1)$ and $(2,2,2,1,\ldots, 1).$ Let $W = \t U.$ Choose $g$ in the intersection $I_n \cap \CC[x_{ij}:ij\in W]$ where $I_n$ is the multiview ideal. So, $g = \sum_{\a} h_\a f_\a$ where each $f_\a$ is a $2-,3-,$or $4-$focal. Note that $\t^{-1}\circ g = \t^{-1}\circ \sum_{\a}h_\a f_\a\in I_n \cap \CC[x_{ij}: ij\in U]$ since $\t^{-1}f_\a$ is still a $2-,3-,4-$focal for each focal $f_\a$ and each permutation $\t.$ Since $\pi_U$ is assumed to be dominant, it follows from Lemma~\ref{lem:nothing supported on U} that $g = 0.$ Thus, it suffices to show that $\pi_U$ is dominant when $U$ is a facet with profile $(3,2,1,1,\ldots, 1)$ or $(2,2,2,1,\ldots, 1),$ as desired.
    \end{proof}

{\em Proof of Theorem \ref{thm:multiview}.} By Lemmas \ref{lem:project-to-C} and \ref{lem:symmetry-reduction}, it is sufficient to check that \eqref{eq:Uprojection} is a dominant map for facets with profile $(3,2,1,1,1,\ldots, 1)$ and $(2,2,2,1,1,\ldots, 1)$.

Assume the statement holds for up to $n-1$ cameras. Now suppose we have $n$ cameras. Without loss of generality, we assume that the $n$th camera contributes one index to a facet $U\in\Delta_n$. Suppose this variable is $x_{nj}$. 
Choose a generic point $p=(p_1, \ldots, p_{n-1},p_n)$ in the affine space $(\CC^1)^U$, where $p_i$ corresponds to an entry in $U$ for camera $i$. Since $p$ is generic, then so is $q = (p_1, \ldots, p_{n-1})$, so that by induction there is a generic point $a = (a_1, \ldots, a_{n-1}) \in M_{n-1}$ that projects to $q$ under $\pi_W$ where $W$ is obtained from $U$ by dropping the last entry. 

Since $a \in M_{n-1}$ there is some point in $\PP^3$ that is imaged to $a_i$ by camera $A_i$ for $1 \leq i \leq n-1$. Let $x(a)\in \CC^4$ be a nonzero point for which
$A_i x(a)$ and $a_i$ are scalar multiples for all $i=1,\ldots , n -1.$ 
Since $(A_1,\ldots , A_n)$ is generic, note that each row of $A_n$ must be nonzero. 
Furthermore, since $a$ is generic, we may assume the $j^{th}$ coordinate of $A_n x(a)$ is some nonzero scalar $c.$
By construction, %
$$(p_1, \ldots, p_{n-1}, c) = \left(\pi_W (a), e_j^T A_n x(a) \right)
\in \pi_U (M_n)$$
Since $M_n$ is a cone in each factor, it follows that $$
(p_1, \ldots, p_{n-1}, p_n)=(p_1, \ldots, p_{n-1}, (p_n/c)\cdot c) \in \pi_U (M_n),
$$
completing the proof that $\pi_U$ is dominant.

We verified %
that the statement holds in the base case of $n=4$ using Macaulay2. Thus, by induction on the number of cameras, the set of $2-,3-, 4-$focals forms a universal Gr\"obner basis of the multiview ideal $I_n$ for any $n\geq 4.$

\section{The universal multiview ideal}
\label{sec:universal multiview ideals}

For $n$ unknown cameras, the universal multiview variety consists of all points $(A_1, \ldots, A_n, p_1, \ldots, p_n)\in (\PP^{3\times 4-1})^n\times (\PP^2)^n$ such that for each $i=1, \ldots, n$ there is some world point $q\in \PP^3$ such that $A_i q = p_i.$ Let $\tM_n$ be the affine cone over the universal multiview variety and $\tI_n$ the vanishing ideal of $\tM_n.$ Note that the difference between $M_n$ and $\tM_n$ is that the former lies in $(\PP^2)^n$ and its vanishing ideal $I_n$ lies in 
$\CC[x_1, \ldots, x_n]$ where $x_i$ consists of the three variables corresponding to the $i$th camera plane, while $\tM_n$ lies in 
$(\PP^{3\times 4-1})^n\times (\PP^2)^n$ and its vanishing ideal $\tI_n$ lies in the polynomial ring with $12n$ variables coming from the $n$ cameras  $A_1, \ldots, A_n$ and the $3n$ variables of the image plane coordinates $x_1, \ldots, x_n$.

In this section, we use the Huang-Larson theorem to show that the set of all $2-,3-, 4-$focals of $n$ unknown camera matrices and image points in the polynomial ring $\CC[\bA, \bx]$ form a universal Gr\"obner basis for the universal multiview ideal $\tI_n.$ It was shown in 
\cite[Theorem 3.2]{atlaspaper} that the  $2-,3-, 4-$focals form a Gr\"obner basis for $\tI_n$ under certain product orders. Thus, in particular, these focals generate $\tI_n$. Our result that the $2-,3-, 4-$focals form a universal Gr\"obner basis of $\tI_n$ strengthens \cite[Theorem 3.2]{atlaspaper}, and is a 
new result.

\subsection{The simplicial complex} For $n\geq 4,$ let $\tD_n$ be the simplicial complex whose nonfaces are generated by the spreads of $2-,3-, 4-$focals in the polynomial ring $\CC[\bA, \bx].$ 
Recall that a $k-$focal is a maximal minor of the matrix \eqref{eq:kfocal matrix}. Both $\D_n$ and $\tD_n$ have nonfaces generated by the spreads of $2-,3-, 4-$focals, except that 
in $\tD_n$ (and $\tI_n$), the entries of $\bA$ are variables. 
Regardless, we expect a relationship between the facets and nonfaces of $\D_n$ and $\tD_n.$ 
We describe a simple process for generating all facets of $\tD_n$ from the facets of $\D_n.$ 

Identify the ground set of  $\tD_n$ with the union of the variables in $\bA$ and $\bx$ coming from all $n$ cameras -- a total of $15n = 12n+3n$ elements.The spread of a polynomial $f \in \CC[\bA,\bx]$ is identified with the set of $\bA$ and $\bx$ variables present in $f$ while profiles, as before, will record the number of $\bf x$ variables from the $n$ cameras.
The following observations will help us understand $\tD_n$.

\begin{lemma}\label{lem:facets-observations}
    If $U\in\tD_n$ is a facet, then the following three conditions hold:
    \begin{enumerate}
        \item If $x_{ij}\notin U$ then $a_{ijk}\in U$ for all values of $k=1,2,3,4$.
        \item There are exactly $n+3$ variables $x_{ij}\in U$ such that $a_{ijk}\in U$ for all values of $k=1,2,3,4$. Moreover, these $\bx-$variables form a facet of $\D_n.$
        \item If $U$ contains more than $n+3$ $\bx-$variables then (2) holds and for each remaining  $x_{ij} \in U$ there is exactly one value of $k$ such that $a_{ijk}\notin U.$
        \end{enumerate}
    It follows that $\tD_n$ is a pure $(13n+2)-$dimensional simplicial complex.
\end{lemma}
\begin{proof}

The nonfaces  of $\tD_n$ are generated by the spreads of $2-,3-,$ and $4-$focals with profiles $(3,3,0,0,\ldots,0),$ $(3,2,2,0\ldots, 0)$ and $(2,2,2,2,\ldots, 0)$ up to permutation of cameras. 
Any such focal, $f$, is the determinant of a submatrix of \eqref{eq:kfocal matrix} with $4+k$ chosen rows. It will be convenient to think of $f$  as a polynomial in all the $\bx$-variables from the chosen rows with 
(symbolic) coefficients the  $4\times 4$ minors of the camera rows appearing in these chosen rows. Note that each  camera variable from the chosen rows appears in at least one $4\times 4$ determinant and hence, in the spread of $f$.  This means that 
\begin{enumerate}
    \item if $x_{ij}\in \spr f,$ then $a_{ijk}\in \spr f$ for all $k=1, 2, 3, 4$, and 
    \item if $x_{ij}\notin\spr f,$ then $a_{ijk}\notin\spr f$ for all $k=1,2,3,4.$
\end{enumerate}
From this, we conclude that if $U\in\tD_n$ then the collection of variables $x_{ij}\in U$ such that $a_{ijk}\in U$ for all $k=1,2,3,4$ must form a face of $\D_n.$ Otherwise, $U$ contains the spread of some $k-$focal. In particular, $U$ has at most $n+3$ such $\bx-$variables.

We now consider the faces of $\tD_n.$ Choose $U\in\tD_n$ and let $U_\bx$ be the collection of $\bx-$variables in $U.$ There are two cases to consider: either $U_\bx\in\D_n$ or $U_\bx\notin\D_n.$

Suppose that $U_\bx\in\D_n.$ Then, $U$ does not contain the $\bx-$spread of a $2-,3-,$ or $4-$focal, and the same is true for $U \cup \{\bA\}$. It follows that $U\cup\{\bA\}\in\tD_n.$ Moreover, if $U$ is a facet of $\tD_n$ then it must be that $U_\bx$ is a facet of $\D_n$. In particular, there are $n+3$ variables $x_{ij}\in U$, as well as all the $\bA$ variables, and  (1) and (2) hold.

Now, suppose that $U_\bx\notin\D_n.$ Assume the $U_\bx$ contains $n+3$ variables. Recall that any facet of $\D_n$ also contains $n+3$ variables and the profile of each facet is nonzero in every entry. It follows that $\type(U_\bx)$ has at least one entry equal to $0$ since any other valid way to arrange $n+3$ $\bx-$variables leads to the profile of a facet. So, there exists some index $i$ such that $x_{ij}\notin U$ for all choices of $j.$ Consider $W=U\cup\{x_{i1}\}.$ Since $U$ does not contain the spread of any $k-$focal, if $W$ contains the spread of the $k-$focal $f$ then the $i^{th}$ entry of $\type(f)$ must be equal to 1. However, this implies that $f$ is a monomial multiple of a generating focal. So, since $U$ does not contain the spread of a focal, neither does $W$ and therefore $W\in\tD_n.$ It follows that $U$ is not a facet of $\tD_n.$ Using the same argument, if $U_\bx$ contains fewer then $n+3$ variables then $U$ is not a facet of $\tD_n.$ We conclude that if $U\in\tD_n$ is a facet, then $U_\bx$ contains at least $n+4$ variables.

    Assume $U_\bx$ contains at least $n+4$ variables. 
    As discussed above, the variables $x_{ij}\in U_\bx$ such that $a_{ijk}\in U$ for all $k=1,2,3,4$ form a face of $\D_n.$ For the remaining $\bx-$variables, there is at least one value of $k$ such that $a_{ijk}\notin U.$ If $U$ is a facet, and thus maximally supported, there are exactly $n+3$ variables in $U_\bx$ such that $a_{ijk}\in U$ for all values $k=1,2,3,4$ and these $\bx-$variables form a facet of $\D_n.$ Moreover, there is at exactly one value of $k$ such that $a_{ijk}\notin U$ for each remaining $\bx-$variable. For each $x_{ij}\notin U$ we have that $a_{ijk}\in U$ for all values of $k=1,2,3,4$ again because  $U$ is maximally supported. %

    Given each of the properties (1), (2), and (3), a counting argument shows that any facet of $\tD_n$ contains $13n+3$ variables and thus $\tD_n$ is a  $(13n+2)-$dimensional pure simplicial complex.
\end{proof}

From the proof of Lemma~\ref{lem:facets-observations} we see that if $U$ is a facet of $\D_n$ then the union $U\cup\{\bA\}$ of $U$ with all variables corresponding to entries in all $n$ camera matrices is a facet of $\tD_n.$ Moreover, every facet of $\tD_n$ has the same size. 
We now 
show that all facets of $\tD_n$ can be found using a simple recursive process.

\begin{theorem}\label{thm:generating-facets}
    Every facet of $\tD_n$ can be found using the following process:
    \begin{enumerate}
        \item Choose a facet $U\in \D_n$ and set $U_0 := U\cup\{\bA\}.$
        \item For $\ell\geq 1,$ choose $x_{ij}\notin U_{\ell-1}$ and set $U_\ell := (U_{\ell-1}\cup\{x_{ij}\})\setminus\{a_{ijk}\}$ for some $a_{ijk}.$
    \end{enumerate}
\end{theorem}
\begin{proof}
    By the previous discussion, $U_0$ is a facet of $\tD_n$ for each choice of $U.$  Moreover,  each $U_\ell$ is a face of $\tD_n.$ Since $\tD_n$ is pure and $U_\ell$ is the same size as $U_0,$ we see that $U_\ell$ is also facet. 
    Thus, we only need to show that the process is exhaustive.

    Choose some facet $W\in\tD_n.$ By Lemma~\ref{lem:facets-observations}, if $x_{ij}\notin W,$ then $a_{ijk}\in W$ for all values of $k$ and there are exactly $n+3$ variables $x_{i_1j_1}, \ldots, x_{i_{n+3}j_{n+3}}\in W$ with $a_{ijk}\in W$ for all $k.$ Set $U=\{x_{i_1j_1}, \ldots, x_{i_{n+3}j_{n+3}}\}$ and $U_0= U\cup\{\bA\}.$ If $W$ contains exactly $n+3$ of the $\bx-$variables, then $U_0=W.$ Suppose then that $W$ contains $n+3+m$ of the $\bx-$variables with $m>0.$ Then, for each $x_{i_{n+3+\ell}j_{n+3+\ell}}\in W$ there exists some value of $k_\ell=1,2,3,4$ such that $a_{ijk_\ell}\notin W.$ So, setting $U_\ell = U_{\ell-1}\cup\{x_{i_{n+3+\ell}j_{n+3+\ell}}\}\setminus\{a_{ijk_\ell}\},$ we see that $W = U_{m}.$ Thus, we have constructed $W$ using the desired process.
\end{proof}

Analogously to~Theorem~\ref{thm:matroid-multiview}, we now describe the
matroid structure of $\widetilde{\Delta}_n.$
\begin{theorem}\label{thm:matroid-universal-multiview}
The simplicial complex $\widetilde{\Delta}_n$ is a  transversal matroid of rank $13n+3$, isomorphic to the union of the uniform matroid $\mathcal{U}_{3,15n}$ on all variables and the direct sum $\mathcal{U}_{13,15}^{\oplus n}$ over the $n$ subsets $\{ x_{i1}, x_{i2}, x_{i3} , a_{11}, \ldots , a_{34}\}$.
\end{theorem}
\begin{proof}
Note first that every facet $U_0=U \cup \{ \mathbf{A} \}$ appearing in~Theorem~\ref{thm:generating-facets} is a basis of the described matroid union. Indeed,~Theorem~\ref{thm:matroid-multiview} implies $U=U' \cup U ''$, where $U '$ is a basis of $\mathcal{U}_{3,3n}$ and $U''$ is a basis of $\mathcal{U}_{1,3}^{\oplus n}$.
Furthermore, since $U''\cup \{ \mathbf{A} \}$ is a basis of $\mathcal{U}_{13,15}^{\oplus n}$ disjoint from $U',$ we deduce that $U_0$ is a basis of
$\mathcal{U}_{3,15n} \vee \mathcal{U}_{13,15}^{\oplus n} $.

Next, observe that any of the facets $U_\ell$ constructed in~Theorem~\ref{thm:generating-facets} is a basis in the matroid union $\mathcal{U}_{3,15n} \vee \mathcal{U}_{13,15}^{\oplus n}.$
For example, we have 
\[
U_1 =\left( U_0 \cup \{ x_{ij} \} \right)\setminus \{ a_{ijk} \} 
=
U' \cup 
\left( U'' \cup \{ \mathbf{A} \} \setminus \{ a_{ijk} \} \cup \{ x_{ij} \}\right),
\]
with $U'$ and $U''$ as in the preceding paragraph, and similarly for $\ell \ge 2.$
Finally, combining observations from Lemma~\ref{lem:facets-observations} and Theorem~\ref{thm:generating-facets}, all facets of the matroid union $\mathcal{U}_{3,15n} \vee \mathcal{U}_{13,15}^{\oplus n}$ have the form $U_\ell$ for some choice of $U_0$ and $\ell .$
\end{proof}

\subsection{Applying the Huang-Larson theorem} Our goal now is to prove that the $2-,3-, 4-$focals form a universal Gr\"obner basis of $\tI_n$. By~Lemma~\ref{lem:project-to-C}, it is sufficient to show that for each facet $U\in\tD_n$ the projection %
$\pi_U:\tM_n \to(\CC^1)^U$ is dominant.
Since the number of facets of $\tD_n$ is large, even in the case where $n=4,$ it is necessary to reduce the dominance checks to a limited number of facets. To this end, we show that we need only consider facets of $\tD_n$ corresponding to facets of $\D_n$ followed by a symmetry reduction analogous to~Lemma~\ref{lem:symmetry-reduction}.

 \begin{lemma}\label{lem:facets-to-check}
    It suffices to prove that $\pi_U$ is dominant for all facets $U$ of the form $U = W \cup\{\bA\}$ where $W$ is a facet of $\D_n.$
 \end{lemma}
\begin{proof}
    Assume that for some facet $U'$ of $\tD_n$ the projection $\pi_{U'}$ is dominant. Let $U = (U'\cup\{x_{ij}\})\setminus \{a_{ijk}\}$ for some values of $i,j,k.$
    
    Choose a generic point $Q \in (\CC^1)^U$. We will argue that there is a point $P \in \tM_n$ such that $\pi_U(P) = Q$. We consider a generic point $Q' \in (\CC^1)^{U'}$ obtained from $Q$ by deleting the entry in position $x_{ij}$ and adding an entry in position $a_{ijk}$. By assumption, there is a point $P' \in \tM_n$ such that $\pi_{U'}(P') = Q'$. 
    The projection $\pi_U(P')$ agrees with $Q$ in all positions except the one indexed by $x_{ij}$. Note that the entry of $P'$ in position $a_{ijk}$ is irrelevant for the projection $\pi_U$ since $U$ does not contain $a_{ijk}$. We use this flexibility to construct $P \in \tM_n$, from $P'$ with $\pi_U(P) = Q$. 

    Suppose $P' = (A_1', \ldots, A_n', p_1', \ldots, p_n')$ and $q_{ij}$ is the entry of $Q$ in position $x_{ij}$. We modify $p_i'$ in the $j$th position to obtain a preimage for $Q$ in $\tM_n$ under $\pi_U$.
Since $P' \in \tM_n$, there is some $y$ and scalars $\lambda_\ell'$ such that $ A_\ell' y = \lambda_\ell' p_\ell'$ for $1 \leq \ell \leq n$. %
Modify $P'$ to $P$ by replacing $A_i'$ with $A_i$ as follows: replace the $jk$ entry of $A_i'$ with $A_{ijk} = \frac{a_{ijk}'y_k +  \lambda_i' (q_{ij} - p_{ij}')}{y_k}.$ Then, $(A_i y)_j = \l_i'q_{ij}$ and so $\pi_{U}(P) = Q$. 
It follows that $\pi_{U} \,:\, \tM_n \rightarrow (\CC^1)^U$ is dominant. 

By Theorem~\ref{thm:generating-facets},  we now need only check that the projections $\pi_U$ are dominant for facets of the form %
$U = W\cup\{\bA\}$ for some facet $W$ of $\D_n,$ as desired.
\end{proof}
The result of Lemma~\ref{lem:facets-to-check} greatly reduces the number of facets that must be considered to check the hypotheses of Theorem~\ref{thm:huang-larson}. We can further reduce the number of facets to be checked using symmetry, as in Lemma~\ref{lem:symmetry-reduction}.
\begin{lemma}\label{lem:symmetry-universal}
    In Lemma~\ref{lem:facets-to-check}, we can further restrict to facets $W$ of $\D_n$ with profile $(3,2,1,1,\ldots, 1)$ and $(2,2,2,1,\ldots, 1).$ 
\end{lemma}
\begin{proof}
   This follows from the same argument as in~Lemma~\ref{lem:symmetry-reduction}
\end{proof}

{\em Proof of~Theorem~\ref{thm:universal-multiview}} By~Lemmas~\ref{lem:symmetry-universal} and~\ref{lem:facets-to-check}, it suffices to check that \eqref{eq:Uprojection} is a dominant map for facets of $\tD_n$ of the form $U=W\cup\{\bA\}$ where $W\in\D_n$ is a facet with profile $(3,2,1,1,\ldots, 1)$ or $(2,2,2,1,\ldots,1).$

Suppose the statement holds for $\tI_{n-1}$ and we now have $n$ cameras. This $n^{th}$ camera contributes $12n$ camera variables, and without loss of generality, one variable $x_{nj}$ to the facet $U\in\tD_n.$ Choose a generic point $Q = (B_1, \ldots, B_n, q_1, \ldots, q_n)\in (\CC^1)^U.$ Since $Q$ is generic, so is the point $Q' = (B_1, \ldots, B_{n-1}, q_1, \ldots, q_{n-1})\in (\CC^1)^{U'}$ where $U'$ is attained from $U$ by removing each variable corresponding to the $n^{th}$ camera. Note that $U'$ is also of the form $U'=W'\cup\{\bA_{n-1}\}$ for a facet $W'\in \D_{n-1}.$ By induction, there exists $(B_1, \ldots, B_{n-1}, p_1, \ldots, p_{n-1})\in\tM_{n-1}$ that projects to $Q'$ under $\pi_{U'}.$ We can assume that the cameras are exactly $B_1, \ldots, B_{n-1}$ since $U'$ contains all camera variables. In particular, there exists $y\in\C^4$ and scalars $\l_i$ such that for $1\leq i\leq n-1,$ we have that %
$B_i y= \l_ip_i.$ Since $B_n, y$ are generic, we may assume that the $j^{th}$ coordinate of $B_ny = c\neq 0.$ By construction,
\begin{align*}
    (B_1, \ldots, B_{n-1}, B_n, q_1, \ldots, q_{n-1}, c) \in \pi_{U}(\tM_n).
\end{align*}
Then, since $\tM_n$ is a cone in each factor, 
\begin{align*}
    Q = (B_1, \ldots, B_n, q_1, \ldots, q_n)  = (B_1, \ldots, B_n, q_1, \ldots, (q_n/c)\cdot c)\in\pi_U(\tM_n),
\end{align*}
completing the proof that $\pi_U$ is dominant.

\section{Conclusion}\label{sec:conclusion}

In conclusion, this paper demonstrates that Theorem~\ref{thm:huang-larson} is a useful tool for proving universal Gr\"{o}bner basis results in the emerging field of \emph{algebraic vision}~\cite{kileel2022snapshot}. 
Specifically, we have obtained a new proof of the previously-known result Theorem~\ref{thm:multiview}, and an entirely new result in Theorem~\ref{thm:universal-multiview}.
We are optimistic that the techniques developed here can be adapted to other problems; in particular, to families of ideals appearing in~\cite{atlaspaper}, such as the resectioning varieties studied in~\cite{connelly2025algebra}.

\bibliographystyle{amsplain}
\bibliography{refs}

\end{document}